\documentclass{article}

\usepackage{amssymb,stmaryrd}
\newcommand{\ent}{\mathrm{ent}}
\renewcommand{\dim}{\mathrm{dim}}
\newcommand{\effdim}{\mathrm{effdim}}
\newcommand{\res}{{\upharpoonright}}
\newcommand{\cat}{{^{\smallfrown}}}
\newcommand{\llb}{\llbracket}
\newcommand{\rrb}{\rrbracket}
\newcommand{\diam}{\mathrm{diam}}
\newcommand{\dom}{\mathrm{dom}}
\newcommand{\rng}{\mathrm{rng}}
\newcommand{\ZZ}{\mathbb{Z}}
\newcommand{\NN}{\mathbb{N}}
\newcommand{\K}{\mathrm{K}}
\newcommand{\KP}{\mathrm{KP}}
\newcommand{\KS}{\mathrm{KS}}
\newcommand{\KM}{\mathrm{KM}}
\newcommand{\KA}{\mathrm{KA}}
\newcommand{\KD}{\mathrm{KD}}
\newcommand{\llex}{<_{\mathrm{lex}}}

\newcommand{\calE}{\mathcal{E}}
\newcommand{\calF}{\mathcal{F}}
\newcommand{\calP}{\mathcal{P}}
\newcommand{\calQ}{\mathcal{Q}}
\newcommand{\calU}{\mathcal{U}}
\newcommand{\calV}{\mathcal{V}}

\usepackage{amsthm}
\theoremstyle{definition}
\newtheorem{thm}{Theorem}[section]
\newtheorem{lem}[thm]{Lemma}
\newtheorem{rem}[thm]{Remark}
\newtheorem{ques}[thm]{Questions}

\begin{document}

\title{Symbolic dynamics:\\
  entropy = dimension = complexity}

\author{Stephen G. Simpson\\
  Department of Mathematics\\
  Vanderbilt University\\
  http://www.math.psu.edu/simpson\\
  sgslogic@gmail.com}

\date{First draft: March 17, 2010\\
  This draft: \today}

\maketitle

\addcontentsline{toc}{section}{Abstract}

\begin{abstract}
  Let $d$ be a positive integer.  Let $G$ be the additive monoid
  $\NN^d$ or the additive group $\ZZ^d$.  Let $A$ be a finite set of
  symbols.  The shift action of $G$ on $A^G$ is given by
  $S^g(x)(h)=x(g+h)$ for all $g,h\in G$ and all $x\in A^G$.  A
  $G$-subshift is defined to be a nonempty closed set $X\subseteq A^G$
  such that $S^g(x)\in X$ for all $g\in G$ and all $x\in X$.  Given a
  $G$-subshift $X$, the topological entropy $\ent(X)$ is defined as
  usual \cite{ruelle-Zd}.  The standard metric on $A^G$ is defined by
  $\rho(x,y)=2^{-|F_n|}$ where $n$ is as large as possible such that
  $x\res F_n=y\res F_n$.  Here $F_n=\{0,1,\ldots,n\}^d$ if $G=\NN^d$,
  and $F_n=\{-n,\ldots,-1,0,1,\ldots,n\}^d$ if $G=\ZZ^d$.  For any
  $X\subseteq A^G$ the Hausdorff dimension $\dim(X)$ and the effective
  Hausdorff dimension $\effdim(X)$ are defined as usual
  \nocite{lc2001}\cite{hausdorff-dim,reimann-phd,re-st-effdim} with
  respect to the standard metric.  It is well known that
  $\effdim(X)=\sup_{x\in X}\liminf_n\K(x\res F_n)/|F_n|$ where $\K$
  denotes Kolmogorov complexity \cite{do-hi-book}.  If $X$ is a
  $G$-subshift, we prove that $\ent(X)=\dim(X)=\effdim(X)$, and
  $\ent(X)\ge\limsup_n\K(x\res F_n)/|F_n|$ for all $x\in X$, and
  $\ent(X)=\lim_n\K(x\res F_n)/|F_n|$ for some $x\in X$.
\end{abstract}

\vfill

{\small


  \medskip

  \noindent MSC2010 Subject Classification: Primary 37B10, Secondary
  37B40, 94A17, 68Q30.

  \medskip

  \noindent We thank the anonymous referees for comments which led to
  improvements in this paper.  In particular, the proof of Lemma
  \ref{lem:limsup} given below was suggested by the first referee and
  is much simpler than our original proof.

  \medskip

  \noindent This paper has been published in \emph{Theory of
    Computing Systems}, 56(3):2015, 527--543.}

\newpage

{\small
\contentsline {section}{Abstract}{1}
\contentsline {section}{\numberline {1}Introduction}{2}
\contentsline {section}{\numberline {2}How this paper came about}{3}
\contentsline {section}{\numberline {3}Background}{4}
\contentsline {subsection}{\numberline {3.1}Topological entropy}{5}
\contentsline {subsection}{\numberline {3.2}Symbolic dynamics}{7}
\contentsline {subsection}{\numberline {3.3}Hausdorff dimension}{8}
\contentsline {subsection}{\numberline {3.4}Kolmogorov complexity}{8}
\contentsline {subsection}{\numberline {3.5}Effective Hausdorff dimension}{9}
\contentsline {subsection}{\numberline {3.6}Measure-theoretic entropy}{10}
\contentsline {section}{\numberline {4}Entropy = dimension}{12}
\contentsline {section}{\numberline {5}Dimension = complexity}{14}
\contentsline {section}{References}{16}
}

\section{Introduction}
\label{sec:intro}

The purpose of this paper is to elucidate a close relationship among
three disparate concepts which are known to play a large role in three
diverse branches of contemporary mathematics.  The concepts are:
\begin{center}
  entropy,\ \ \ \ Hausdorff dimension,\ \ \ \ Kolmogorov complexity.
\end{center}
Some relationships among these concepts are well known; see for
instance \cite{brudno,pesin-dim,brudno-white}.  Nevertheless, it seems
to us that the full depth of the relationships has been insufficiently
appreciated.  Below we prove that, in an important special case, all
three concepts coincide.

Here is a brief overview of the above-mentioned concepts.

\begin{enumerate}
\item \emph{Hausdorff dimension} is a basic concept in metric
  geometry.  See for instance the original paper by Hausdorff
  \cite{hausdorff-dim} and the classic treatise by C. A. Rogers
  \cite{c-a-rogers}.  To any set $X$ in a metric space one assigns a
  nonnegative real number $\dim(X)=$ the Hausdorff dimension of $X$.
  In the case of smooth sets such as algebraic curves and surfaces,
  the Hausdorff dimension is a nonnegative integer and coincides with
  other familiar notions of dimension from algebra, differential
  geometry, etc.  For example, the Hausdorff dimension of a smooth
  surface in $n$-dimensional space is $2$.  On the other hand,
  Hausdorff dimension applies also to non-smooth sets with nonintegral
  dimension, e.g., fractals and Julia sets \cite{falconer}.
\item \emph{Kolmogorov complexity} plays an important role in
  information theory \cite{cover-thomas,shannon-weaver}, theoretical
  computer science \cite{li-vitanyi,us-sh}, and
  recursion/computability theory \cite{do-hi-book,nies-book}.  To each
  finite mathematical object $\tau$ one assigns a nonnegative integer
  $\K(\tau)=$ the complexity of $\tau$.  Roughly speaking, $\K(\tau)$
  is the length in bits of the shortest computer program which
  describes $\tau$.  In this sense $\K(\tau)$ measures the ``amount of
  information'' which is inherent in $\tau$.
\item \emph{Entropy} is an important concept in dynamical systems
  theory \cite{de-gr-si}.  Classically, a dynamical system consists of
  a set $X$ together with a mapping $T:X\to X$ and one studies the
  long-term behavior of the orbits $\langle T^n(x)\mid
  n=0,1,2,\ldots\rangle$ for each $x\in X$.  More generally, one
  considers an action $T$ of a group or semigroup $G$ on a set $X$,
  and then the orbit of $x\in X$ is $\langle T^g(x)\mid g\in
  G\rangle$.  The entropy of the system $X,T$ is a nonnegative real
  number which has a rather complicated definition but is intended to
  quantify the ``exponential growth rate'' of the system.

  An especially useful class of dynamical systems are the symbolic
  systems, a.k.a., subshifts
  \cite{ho-me-tilings,lind-marcus,shields-erg,2dim}.  Given a finite
  set of symbols $A$, one defines the shift action of $G$ on $A^G$ as
  usual.  A subshift is then defined to be a closed, shift-invariant
  subset of $A^G$.  These symbolic systems play a large role in
  general dynamical systems theory, because for any dynamical system
  $X,T$ one can consider partitions $\pi:X\to A$ and then the behavior
  of an orbit $\langle T^g(x)\mid g\in G\rangle$ is reflected by the
  behavior of its ``symbolic trace,'' $\langle\pi(T^g(x))\mid g\in
  G\rangle$, which is a point in $A^G$.
\end{enumerate}

Our main results in this paper are Theorems \ref{thm:G} and
\ref{thm:dim-k} below.  They say the following.  Let $d$ be a positive
integer, let $G$ be the additive monoid $\NN^d$ or the additive group
$\ZZ^d$, let $A$ be a finite set of symbols, and let $X\subseteq A^G$
be a subshift.  Then, the entropy of $X$ is equal to the Hausdorff
dimension of $X$ with respect to the standard metric on $A^G$.
Moreover, the entropy of $X$ has a sharp characterization in terms of
the Kolmogorov complexity of the finite configurations which occur in
the orbits of $X$.

In connection with the characterization of entropy in terms of
Kolmogorov complexity, it is interesting to note that both of these
concepts originated with A. N. Kolmogorov, but in different contexts
\cite{ks-entropy,k-original}.

\section{How this paper came about}
\label{sec:ack}

This paper is an outcome of my reading and collaboration over the past
several years.  Here are some personal comments on that process.

It began with my study of Bowen's alternative definition of
topological entropy \cite[pages 125--126]{bowen-tams-73}.  Obviously
Bowen's definition resembles the standard definition of Hausdorff
dimension in a metric space, and this led me to consider the following
question:
\begin{quote}
  Given a subshift $X$, what is the precise relationship between the
  topological entropy of $X$ and the Hausdorff dimension of $X$?
\end{quote}
Specifically, let $A$ be a finite set of symbols.  From
\cite[Proposition 1]{bowen-tams-73} it was clear to me that the
topological entropy of a one-sided subshift $X\subseteq A^\NN$ is
equal to the Hausdorff dimension of $X$ with respect to the standard
metric.  And eventually I learned that this result appears explicitly
in Furstenberg 1967 \cite[Proposition III.1]{furst-tcs-67}.  But what
about other kinds of subshifts on $A$?  For instance, what about the
two-sided case, i.e., subshifts in $A^\ZZ$ or $A^{\ZZ^d}$ or more
generally $A^G$ where $G$ is a countable amenable group
\nocite{lmsln310}\cite{amenable-survey}?  And what about the general
one-sided case, i.e., subshifts in $A^{\NN^d}$ or more generally $A^G$
where $G$ is countable amenable semigroup, whatever that may mean?

During February, March and April of 2010 I discussed these issues with
several colleagues: John Clemens, Vaughn Climenhaga \cite[Example
4.1]{climenhaga-10}, Manfred Denker \cite{de-gr-si}, Michael Hochman
\cite{ho-me-tilings}, Anatole Katok \cite{katok-thouvenot}, Daniel
Mauldin, Yakov Pesin \cite{pesin-dim}, Jan Reimann \cite{reimann-phd},
Alexander Shen \cite{us-sh}, Daniel Thompson, Jean-Paul Thouvenot
\cite{katok-thouvenot}.  All of these discussions were extremely
helpful.  In particular, Hochman and Mauldin provided several ideas
which play an essential role in this paper.

\section{Background}
\label{sec:background}

In this section we present some background material concerning
symbolic dynamics, entropy, Hausdorff dimension, and Kolmogorov
complexity.  All of the concepts and results in this section are well
known.

We write
\begin{center}
  $\NN=\{0,1,2,\ldots\}=\{$the nonnegative integers$\}$
\end{center}
and
\begin{center}
  $\ZZ=\{\ldots,-2,-1,0,1,2,\ldots\}=\{$the integers$\}$.
\end{center}
Throughout this paper, let $G$ be the additive monoid $\NN^d$ or the
additive group $\ZZ^d$ where $d$ is a fixed positive integer.  An
\emph{action} of $G$ on a set $X$ is a mapping $T:G\times X\to X$ such
that $T^e(x)=x$ and $T^g(T^h(x))=T^{g+h}(x)$ for all $g,h\in G$ and
all $x\in X$.  Here $e$ is the identity element of $G$.  It is useful
to write $G$ in a specific\footnote{In particular, the sequence $F_n$
  with $n=0,1,2,\ldots$ is a F{\o}lner sequence for $G$.} way as the
union of a sequence of finite sets, namely $G=\bigcup_{n=0}^\infty
F_n$ where $F_n=\{0,1,\ldots,n\}^d$ if $G=\NN^d$, and
$F_n=\{-n,\ldots,-1,0,1,\ldots,n\}^d$ if $G=\ZZ^d$.  In particular we
have $F_0=\{0\}^d=\{e\}$.  We also write $F_{-1}=\emptyset=$ the empty
set.  For any finite set $F$ we write $|F|=$ the cardinality of $F$.
For any function $\Phi$ we write $\dom(\Phi)=$ the domain of $\Phi$,
and $\rng(\Phi)=$ the range of $\Phi$, and
\begin{center}
  $\Phi:\,\subseteq X\to Y$
\end{center}
meaning that $\Phi$ is a function with $\dom(\Phi)\subseteq X$ and
$\rng(\Phi)\subseteq Y$.  Apart from this, all of our set-theoretic
notation is standard.

\subsection{Topological entropy}
\label{sec:ent-t}

We endow $G$ with the discrete topology.  Let $X$ be a nonempty
compact set in a topological space, and let $T:G\times X\to X$ be a
continuous action of $G$ on $X$.  The ordered pair $X,T$ is called a
\emph{compact dynamical system}.  We now define the topological
entropy of $X,T$.

An \emph{open cover of $X$} is a set $\calU$ of open sets such that
$X\subseteq\bigcup\calU$.  In this case we write
\begin{center}
  $C(X,\calU)=\min\{|\calF|\mid\calF\subseteq\calU,X\subseteq\bigcup\calF\}$.
\end{center}
Note that $C(X,\calU)$ is a positive integer.  If $\calU$ and $\calV$
are open covers of $X$, then
\begin{center}
  $\sup(\calU,\calV)=\{U\cap V\mid U\in\calU,V\in\calV\}$
\end{center}
is again an open cover of $X$, and
\begin{center}
  $C(X,\sup(\calU,\calV))\le C(X,\calU)C(X,\calV)$.
\end{center}
For each $g\in G$ and each open cover $\calU$ of $X$, we have another
open cover $\calU^g=(T^g)^{-1}(\calU)=\{(T^g)^{-1}(U)\mid
U\in\calU\}$.  Hence, for each finite set $F\subset G$ we have an open
cover $\calU^F=\sup\{\calU^g\mid g\in F\}$.  Let us write
$C(X,T,\calU,F)=C(X,\calU^F)$.  Note that $C(X,\calU^g)\le
C(X,\calU)$, hence $C(X,\calU^F)\le C(X,\calU)^{|F|}$, hence
$\log_2C(X,T,\calU,F)\le|F|\log_2C(X,\calU)$.  We define
\begin{equation}
  \label{eq:entXTU}
  \ent(X,T,\calU)=\lim_{n\to\infty}\frac{\log_2C(X,T,\calU,F_n)}{|F_n|}
\end{equation}
and
\[
\ent(X,T)=\sup\{\ent(X,T,\calU)\mid\calU\hbox{ is an open cover of }X\}\,.
\]
The nonnegative real number $\ent(X,T)$ is known as the
\emph{topological entropy}\footnote{Instead of $\log_2$ we could use
  $\log_b$ for any fixed $b>1$, for instance $b=e$ or $b=10$.  The
  base $b=2$ is convenient for information theory, where entropy is
  measured in bits.} of $X,T$.  It measures what might be called the
``asymptotic exponential growth rate'' of $X,T$.  See for instance
\cite{de-gr-si,mis-Nd,ruelle-Zd}.

\begin{lem}
  The limit in equation (\ref{eq:entXTU}) exists.
\end{lem}

\begin{proof}
  Let us write $C_n=C(X,T,\calU,F_n)$.  Clearly $C_m\le C_n$ whenever
  $m\le n$.  Moreover, it is easy to see that $C_{nk}\le C_n^{k^d}$
  for all positive integers $k$.  We are trying to prove that
  $\log_2C_n/|F_n|$ approaches a limit as $n\to\infty$.  Assume
  $G=\ZZ^d$, so that $|F_n|=(2n+1)^d$.  (The case $G=\NN^d$ is
  similar, with $|F_n|=(n+1)^d$.)

  Fix a positive integer $m$.  Given $n\ge m$, let $k$ be a positive
  integer such that $mk\le n<m(k+1)$.  We have $|F_n|\ge|F_{mk}|$ and
  \[
  \frac{|F_{mk}|}{k^d|F_m|}=\left(\frac{2mk+1}{2mk+k}\right)^d
  >\left(\frac{2m}{2m+1}\right)^d
  \]
  and $\log_2C_n\le\log_2C_{m(k+1)}\le(k+1)^d\log_2C_m$, hence
  \[
  \frac{\log_2C_n}{|F_n|}\le\frac{(k+1)^d\log_2C_m}{|F_{mk}|}
  \le\frac{(k+1)^d\log_2C_m}{k^d|F_m|}\left(\frac{2m+1}{2m}\right)^d.
  \]
  As $n\to\infty$ we have $k\to\infty$, hence
  \[
  \limsup_{n\to\infty}\frac{\log_2C_n}{|F_n|}
  \le\frac{\log_2C_m}{|F_m|}\left(\frac{2m+1}{2m}\right)^d,
  \]
  and this holds for all $m$, hence
  \[
  \limsup_{n\to\infty}\frac{\log_2C_n}{|F_n|}
  \le\liminf_{m\to\infty}\frac{\log_2C_m}{|F_m|}\,.
  \]
  In other words, $\lim_{n\to\infty}\log_2C_n/|F_n|$ exists, Q.E.D.
\end{proof}

Let $\calU$ and $\calV$ be open covers of $X$.  We say that $\calU$
\emph{refines} $\calV$ if each $U\in\calU$ is included in some
$V\in\calV$.  Obviously this implies $C(X,\calU)\ge C(X,\calV)$, and
it is also easy to see that $\ent(X,T,\calU)\ge\ent(X,T,\calV)$.

\begin{lem}
  \label{lem:entUF}
  For each $m$ we have $\ent(X,T,\calU^{F_m})=\ent(X,T,\calU)$.
\end{lem}

\begin{proof}
  Clearly $\calU^{F_m}$ refines $\calU$, hence
  $\ent(X,T,\calU^{F_m})\ge\ent(X,T,\calU)$.  For all $n$ we have
  $F_{m+n}=F_m+F_n$, hence
  $\calU^{F_{m+n}}\subseteq(\calU^{F_m})^{F_n}$, hence
  $\calU^{F_{m+n}}$ refines $(\calU^{F_m})^{F_n}$, hence
  $C(X,\calU^{F_{m+n}})\ge C(X,(\calU^{F_m})^{F_n})$, hence
  $C(X,T,\calU,F_{m+n})\ge C(X,T,\calU^{F_m},F_n)$, hence
  \[
  \frac{\log_2C(X,T,\calU,F_{m+n})}{|F_n|}
  \ge\frac{\log_2C(X,T,\calU^{F_m},F_n)}{|F_n|}\,,
  \]
  hence
  \[
  \frac{\log_2C(X,T,\calU,F_{m+n})}{|F_{m+n}|}\cdot\frac{|F_{m+n}|}{|F_n|}
  \ge\frac{\log_2C(X,T,\calU^{F_m},F_n)}{|F_n|}\,.
  \]
  Taking the limit as $n\to\infty$ and noting that
  \[
  \lim_{n\to\infty}\frac{|F_{m+n}|}{|F_n|}=1\,,
  \]
  we see that $\ent(X,T,\calU)\ge\ent(X,T,\calU^{F_m})$.  This
  completes the proof.
\end{proof}

Let $\calU$ be an open cover of $X$.  We say that $\calU$ is a
\emph{topological generator} if for each open cover $\calV$ of $X$
there exists $m$ such that $\calU^{F_m}$ refines $\calV$.  The
following theorem says that we can use a topological generator to
compute $\ent(X,T)$.

\begin{thm}
  \label{thm:topgen}
  If $\calU$ is a topological generator, then
  $\ent(X,T)=\ent(X,T,\calU)$.
\end{thm}

\begin{proof}
  Let $\calV$ be an open cover of $X$.  Since $\calU$ is a topological
  generator, let $m$ be such that $\calU^{F_m}$ refines $\calV$.  Then
  $\ent(X,T,\calU^{F_m})\ge\ent(X,T,\calV)$, so by Lemma
  \ref{lem:entUF} we have $\ent(X,T,\calU)\ge\ent(X,T,\calV)$.  Thus
  $\ent(X,T,\calU)=\ent(X,T)$.
\end{proof}

\subsection{Symbolic dynamics}
\label{sec:sd}

An important class of dynamical systems are the \emph{symbolic}
dynamical systems, also known as \emph{subshifts}.  We now present
some background material on subshifts.  See also
\cite[\S13.10]{lind-marcus} and
\nocite{cm-469}\cite{boyle-open,ho-me-tilings,shields-erg,2dim}.

As before, let $d$ be a positive integer, and let $G$ be the additive
monoid $\NN^d$ or the additive group $\ZZ^d$.  Let $A$ be a nonempty
finite set of symbols.  We endow $A$ with the discrete topology.  Let
$A^G=\{x\mid x:G\to A\}$.  We endow $A^G$ with the product topology.
Note that each $x\in A^G$ is a function from $G$ to $A$.  For each
finite set $F\subset G$ and each $x\in A^G$ let $x\res F$ be the
restriction of $x$ to $F$.  Thus $A^F=\{x\res F\mid x\in A^G\}$.  For
each $\sigma\in A^F$ we write $\dom(\sigma)=F$ and $|\sigma|=|F|$ and
$\llb\sigma\rrb=\{x\in A^G\mid x\res F=\sigma\}$.  Note that
$\llb\sigma\rrb$ is a nonempty clopen set in $A^G$, and
$\{\llb\sigma\rrb\mid\sigma\in A^F\}$ is a pairwise disjoint covering
of $A^G$.  Let $A^*=\bigcup_{n=0}^\infty A^{F_n}$ and note that
$\{\llb\sigma\rrb\mid\sigma\in A^*\}$ is a basis for the topology of
$A^G$.  For any $T\subseteq A^*$ we write $\llb
T\rrb=\bigcup_{\sigma\in T}\llb\sigma\rrb$.  Thus $\llb T\rrb$ is an
open set in $A^G$.

The \emph{shift action} of $G$ on $A^G$ is the mapping $S:G\times
A^G\to A^G$ given by $S^g(x)(h)=x(g+h)$ for all $g,h\in G$ and all
$x\in A^G$.  Thus $A^G,S$ is a compact dynamical system, known as the
\emph{full shift}.  Since $F_0=\{0\}^d$ is a singleton set, there is
an obvious one-to-one correspondence between $A^{F_0}$ and $A$, so we
identify $A^{F_0}$ with $A$.  The \emph{canonical open covering} of
$A^G$ is $\calU=\calU(A,G)=\{\llb a\rrb\mid a\in A\}$.  For each
finite set $F\subset G$ we have $\calU^F=\{\llb\sigma\rrb\mid\sigma\in
A^F\}$.  By compactness of $A^G$ it follows that $\calU$ is a
topological generator.  Moreover
$C(A^G,S,\calU,F)=C(A^G,\calU^F)=|\calU^F|=|A^F|=|A|^{|F|}$, so by
Theorem \ref{thm:topgen} we have
$\ent(A^G,S)=\ent(A^G,S,\calU)=\log_2|A|$.

A set $X\subseteq A^G$ is said to be \emph{shift-invariant} if
$S^g(x)\in X$ for all $g\in G$ and all $x\in X$.  A \emph{subshift} is
a nonempty, closed, shift-invariant subset of $A^G$.  Each subshift
$X\subseteq A^G$ gives rise to a compact dynamical system $X,S\res
G\times X$.  We write $\ent(X)=\ent(X,S\res G\times X)$, etc.  Since
$\calU^F$ is a pairwise disjoint covering of $A^G$, we have
$C(X,\calU,F)=C(X,\calU^F)=|X\res F|$ where $X\res F=\{x\res F\mid
x\in X\}$.  Since $\calU$ is a topological generator, it follows by
Theorem \ref{thm:topgen} that
\begin{equation}
  \label{eq:entX}
  \ent(X)=\lim_{n\to\infty}\frac{\log_2|X\res F_n|}{|F_n|}\,.
\end{equation}

\begin{lem}
  \label{lem:entX}
  We have
  \[
  \lim_{n\to\infty}|X\res F_n|2^{-s|F_n|}=\left\{
    \begin{array}{ll}
      0 & \hbox{ if }s>\ent(X)\,,\\[6pt]
      \infty & \hbox{ if }s<\ent(X)\,.
    \end{array}
    \right.
    \]
\end{lem}

\begin{proof}
  First suppose $s>\ent(X)$.  Fix $\epsilon>0$ such that
  $s-\epsilon>\ent(X)$.  Equation (\ref{eq:entX}) implies that for all
  sufficiently large $n$ we have $(s-\epsilon)|F_n|>\log_2|X\res
  F_n|$, hence $|X\res F_n|2^{-|F_n|s}<2^{-\epsilon|F_n|}$.  Letting
  $n\to\infty$ we have $|F_n|\to\infty$, hence
  $\lim_{n\to\infty}|X\res F_n|2^{-s|F_n|}=0$.

  Next, suppose $s<\ent(X)$.  Fix $\epsilon>0$ such that
  $s+\epsilon<\ent(X)$.  Equation (\ref{eq:entX}) implies that for all
  sufficiently large $n$ we have $(s+\epsilon)|F_n|<\log_2|X\res
  F_n|$, hence $|X\res F_n|2^{-|F_n|s}>2^{\epsilon|F_n|}$.  Letting
  $n\to\infty$ we have $|F_n|\to\infty$, hence
  $\lim_{n\to\infty}|X\res F_n|2^{-s|F_n|}=\infty$.
\end{proof}

\subsection{Hausdorff dimension}
\label{sec:dim}

Let $X$ be a set in a metric space.  The \emph{$s$-dimensional
  Hausdorff measure} of $X$ is defined as
\[
\mu_s(X)\,\,\,=\,\,\,
\lim_{\epsilon\to0}\,\,\inf_\calE\,\,\sum_{E\in\calE}\diam(E)^s
\]
where $\diam(E)$ is the diameter of $E$.  Here $\calE$ ranges over
coverings of $X$ with the property that $\diam(E)\le\epsilon$ for all
$E\in\calE$.  The \emph{Hausdorff dimension} of $X$ is
\[
\dim(X)=\inf\{s\mid\mu_s(X)=0\}.
\]
Hausdorff measures and Hausdorff dimension have been widely studied,
e.g., in connection with the geometry of fractals
\cite{falconer,hausdorff-dim,c-a-rogers}.

We now define what we mean by the Hausdorff dimension of a subshift.
The \emph{standard metric} on $A^G$ is given by $\rho(x,y)=2^{-|F_n|}$
where $n=-1,0,1,2,\ldots$ is as large as possible such that $x\res
F_n=y\res F_n$.  (Recall that $F_{-1}=\emptyset$.)  Clearly the
standard metric on $A^G$ induces the product topology on $A^G$.
Moreover, the standard metric is an \emph{ultrametric}, i.e.,
$\rho(x,y)\le\max(\rho(x,z),\rho(y,z))$ for all $x,y,z$.  For any set
$X\subseteq A^G$ we define $\dim(X)=$ the Hausdorff dimension of $X$
with respect to the standard metric on $A^G$.

\begin{lem}
  \label{lem:easy}
  For all subshifts $X\subseteq A^G$ we have $\ent(X)\ge\dim(X)$.
\end{lem}

\begin{proof}
  For each $E\subseteq A^G$ we have $\diam(E)\le2^{-|F_n|}$ if and
  only if $E\subseteq\llb\sigma\rrb$ for some $\sigma\in A^{F_n}$.
  Therefore, in the definition of $\mu_s(X)$ and $\dim(X)$ for an
  arbitrary set $X\subseteq A^G$, we may safely assume that each $E$
  is a basic open set, i.e., $E=\llb\sigma\rrb$ for some $\sigma\in
  A^*$.  Moreover, for each $\sigma\in A^*$ we have
  $\diam(\llb\sigma\rrb)=2^{-|\sigma|}$.

  Assume now that $X$ is a subshift, and suppose $s>\ent(X)$.  By
  Lemma \ref{lem:entX} we have
  \begin{equation}
    \label{eq:Fns0}
    \lim_{n\to\infty}|X\res F_n|2^{-|F_n|s}=0\,.
  \end{equation}
  But for each $n$ we have $X\subseteq\bigcup_{x\in X}\llb x\res
  F_n\rrb$ and $\diam(\llb x\res F_n\rrb)=2^{-|F_n|}$, so
  (\ref{eq:Fns0}) implies that $\mu_s(X)=0$, hence $s\ge\dim(X)$.
  Since this holds for all $s>\ent(X)$, it follows that
  $\ent(X)\ge\dim(X)$.
\end{proof}

\begin{rem}
  In \S\ref{sec:ent-dim} we shall prove that for all subshifts
  $X\subseteq A^G$, $\ent(X)=\dim(X)$.  In other words, the
  topological entropy of a subshift is equal to its Hausdorff
  dimension with respect to the standard metric.  While the special
  case $G=\NN$ is due to Furstenberg \cite[Proposition
  III.1]{furst-tcs-67}, the general result for $G=\NN^d$ or $G=\ZZ^d$
  appears to be new.
\end{rem}

\subsection{Kolmogorov complexity}
\label{sec:k}

We now present some background material on Kolmogorov complexity.

As in \S\ref{sec:sd} let $A^*=\bigcup_{n=0}^\infty A^{F_n}$.  In
addition let $\{0,1\}^*$ be the set of finite sequences of $0$'s and
$1$'s.  For each Turing machine $M$ and each finite sequence
$\alpha\in\{0,1\}^*$, let $M(\alpha)$ be the run of $M$ with input
$\alpha$.  A function
\begin{center}
  $\Phi:\,\subseteq\{0,1\}^*\to A^*$
\end{center}
is said to be \emph{partial computable} if there exists a Turing
machine $M$ such that for all $\alpha\in\{0,1\}^*$,
$\alpha\in\dom(\Phi)$ if and only if $M(\alpha)$ eventually halts, in
which case it halts with output $\Phi(\alpha)$.  For each such $\Phi$
and each $\xi\in A^*$ let
\begin{center}
  $\K_\Phi(\xi)=\min(\{|\alpha|\mid\Phi(\alpha)=\xi\}\cup\{\infty\})\,$.
\end{center}
A partial computable function $\Psi:\,\subseteq\{0,1\}^*\to A^*$ is
said to be \emph{universal} if for each partial computable function
$\Phi:\,\subseteq\{0,1\}^*\to A^*$ there exists a constant $c$ such
that for all $\xi\in A^*$ we have $\K_\Psi(\xi)\le\K_\Phi(\xi)+c$.
The existence of such a universal function is easily proved.  Fix such
a universal function $\Psi$.  For each $\xi\in A^*$ we define the
\emph{Kolmogorov complexity} of $\xi$ to be $\K(\xi)=\K_\Psi(\xi)$.
Note that $\K(\xi)$ is well defined up to an additive constant, i.e.,
up to $\pm\,O(1)$.  Here ``well defined'' means that $\K(\xi)$ is
independent of the choice of $\Psi$.

\begin{rem}
  Actually the complexity notion $\K$ defined above is only one of
  several variant notions, denoted in \cite{us-sh} as $\KP$, $\KS$,
  $\KM$, $\KA$, $\KD$.  These variants are useful in many contexts
  \cite{do-hi-book}.  However, for our purposes in this paper, the
  differences among them are immaterial.
\end{rem}

\subsection{Effective Hausdorff dimension}
\label{sec:effdim}

We now present some background material concerning the effective or
computable variant of Hausdorff dimension.  Throughout this paper the
words ``effective'' and ``computable'' refer to Turing's theory of
computability and unsolvability \cite{rogers,turing-36}.

A \emph{Polish space} is a complete separable metric space.  An
\emph{effectively presented Polish space} consists of a Polish space
$Z,\rho$ together with a mapping $\Phi:\NN\to Z$ such that
$\rng(\Phi)$ is dense in $Z,\rho$ and the real-valued function
$(m,n)\mapsto\rho(\Phi(m),\Phi(n)):\NN\times\NN\to[0,\infty)$ is
computable.  In this case we define the \emph{basic open sets} of
$Z,\rho,\Phi$ to be those of the form
\begin{center}
  $B(n,r)=\{x\in Z\mid\rho(\Phi(n),x)<r\}$
\end{center}
where $r$ is a positive rational number and $n\in\NN$.  A sequence of
basic open sets $B_i$, $i=1,2,\ldots$ is said to be \emph{computable}
if there exist computable sequences $n_i$, $r_i$, $i=1,2,\ldots$ such
that $B_i=B(n_i,r_i)$ for all $i$.  A set $X\subseteq Z$ is said to be
\emph{effectively closed} if its complement $Z\setminus X$ is
\emph{effectively open}, i.e., $Z\setminus X=\emptyset$ or $Z\setminus
X=\bigcup_{i=1}^\infty B_i$ where $B_i$, $i=1,2,\ldots$ is a
computable sequence of basic open sets.  We say that $X$ is
\emph{effectively compact} if it is effectively closed and
\emph{effectively totally bounded}, i.e., there exists a computable
function $f:\NN\to\NN$ such that
$X\subseteq\bigcup_{n=1}^{f(i)}B(n,2^{-i})$ for each $i$.

Let $s$ be a positive real number.  We say that $X$ is
\emph{effectively $s$-null} if there exists a computable double
sequence of basic open sets $B_{ij}$, $i,j=1,2,\ldots$, such that
$X\subseteq\bigcup_{j=1}^\infty B_{ij}$ and
$\sum_{j=1}^\infty\diam(B_{ij})^s\le2^{-i}$ for each $i$.  The
\emph{effective Hausdorff dimension of $X$} is defined as
\[
\effdim(X)\,=\,\inf\{s\mid X\hbox{ is effectively }s\hbox{-null}\}\,.
\]
Note that, although the Hausdorff dimension of a singleton point
$\{x\}$ is always 0, there may be no computable way to ``observe''
this, so the effective Hausdorff dimension of a noncomputable point
may be $>0$.  In fact, for any set $X$ one has
\begin{equation}
  \label{eq:effdim}
  \effdim(X)\,\,=\,\,\sup_{x\in X}\,\effdim(\{x\})\,.
\end{equation}
On the other hand, it is known that $\effdim(X)=\dim(X)$ provided $X$
is effectively compact.  See for instance \cite[Chapter 13]{do-hi-book} and
\nocite{lc2001,9asian}\cite{reimann-phd,re-st-effdim,re-st-tests}.

The above definitions and remarks apply to the effectively compact,
effectively presented\footnote{Our $\Phi$ for $A^G$ is obtained as
  follows.  Let $\#:A^*\to\NN$ be a \emph{standard G\"odel numbering}
  of $A^*$.  In other words, for each $\sigma\in A^*$ let $\#(\sigma)$
  be a numerical code for $\sigma$ from which $\sigma$ can be
  effectively recovered.  Let $a$ be a fixed symbol in $A$.  Define
  $\Phi:\NN\to A^G$ by letting $\Phi(\#(\sigma))=x_\sigma\in A^G$
  where $x_\sigma\in\llb\sigma\rrb$ and $x_\sigma(g)=a$ for all $g\in
  G\setminus\dom(\sigma)$.}  Polish space $A^G$ with the standard
metric as defined in \S\ref{sec:dim}.  In particular we have
$\effdim(X)=\dim(X)$ for all effectively closed sets $X\subseteq A^G$.
In \S\ref{sec:dim-k} below we shall prove that $\effdim(X)=\dim(X)$
for all subshifts $X\subseteq A^G$.  This result holds even if $X$ is
not effectively closed.

For arbitrary subsets of $A^G$, the following theorem exhibits a
relationship between effective Hausdorff dimension and Kolmogorov
complexity.  We shall see in Theorem \ref{thm:dim-k} that the
relationship is even closer when $X$ is a subshift.

\begin{thm}[Mayordomo's Theorem]
  \label{thm:effdim-k}
  For any set $X\subseteq A^G$ we have
  \[
  \effdim(X)\,\,=\,\,\sup_{x\in
    X}\,\liminf_{n\to\infty}\,\frac{\K(x\res F_n)}{|F_n|}\,.
  \]
\end{thm}

\begin{proof}
  This follows from (\ref{eq:effdim}) together with \cite[Theorem
  13.3.4]{do-hi-book}.
\end{proof}

\subsection{Measure-theoretic entropy}
\label{sec:ent-m}

We now present some background material on measure-theoretic entropy.
We state two important theorems without proof but with references to
the literature.

Let $X,\mu$ be a probability space.  An action $T:G\times X\to X$ is
said to be \emph{measure-preserving} if $\mu((T^g)^{-1}(P))=\mu(P)$
for each $g\in G$ and each $\mu$-measurable set $P\subseteq X$.  In
this case the ordered triple $X,T,\mu$ is called a
\emph{measure-theoretic dynamical system}.  We now proceed to define
the measure-theoretic entropy of $X,T,\mu$.

A \emph{measurable partition} of $X$ is a finite set $\calP$ of
pairwise disjoint $\mu$-measurable subsets of $X$ such that
$X=\bigcup\calP$.  In this case we write
\[
H(X,\mu,\calP)=-\sum_{P\in\calP}\mu(P)\log_2\mu(P)\,.
\]
If $\calP$ and $\calQ$ are measurable
partitions of $X$, then
\[
\sup(\calP,\calQ)=\{P\cap Q\mid P\in\calP,Q\in\calQ\}
\]
is again a measurable partition of $X$, and it can be shown
\cite[10.4(d)]{de-gr-si} that
\begin{equation}
  \label{eq:HPQ}
  H(X,\mu,\sup(\calP,\calQ))\le H(X,\mu,\calP)+H(X,\mu,\calQ)\,.
\end{equation}
For each $g\in G$ and each measurable partition $\calP$ of $X$, we
have another measurable partition
$\calP^g=(T^g)^{-1}(\calP)=\{(T^g)^{-1}(P)\mid P\in\calP\}$.  Hence,
for each finite set $F\subset G$ we have a measurable partition
$\calP^F=\sup\{\calP^g\mid g\in F\}$.  Let us write
$H(X,T,\mu,\calP,F)=H(X,\mu,\calP^F)$.  It follows from (\ref{eq:HPQ})
that $H(X,T,\mu,\calP,F)\le|F|H(X,\mu,\calP)$.  We define
\begin{equation}
  \label{eq:entXmuP}
  \ent(X,T,\mu,\calP)=\lim_{n\to\infty}\frac{H(X,T,\mu,\calP,F_n)}{|F_n|}
\end{equation}
and
\[
\ent(X,T,\mu)=\sup\{\ent(X,T,\mu,\calP)\mid\calP\hbox{ is a measurable
  partition of }X\}.
\]
It can be proved that the limit in (\ref{eq:entXmuP}) exists.  The
nonnegative real number $\ent(X,T,\mu)$ is known as the
\emph{measure-theoretic entropy} of $X,T,\mu$.  It plays an important
role in ergodic theory.  See for instance
\cite{de-gr-si,mis-Nd,ow-amenable}.

Let $X,T,\mu$ be a measure-theoretic dynamical system.  A set
$P\subseteq X$ is said to be \emph{$G$-invariant} if
$(T^g)^{-1}(P)\subseteq P$ for all $g\in G$.  The system $X,T,\mu$ is
said to be \emph{ergodic} if for every $G$-invariant $\mu$-measurable
set $P\subseteq X$ we have $\mu(P)=0$ or $\mu(P)=1$.

Now let $d$ be a positive integer, let $G=\NN^d$ or $\ZZ^d$, let $A$
be a nonempty finite set of symbols, and let $X\subseteq A^G$ be a
subshift.  A Borel probability measure $\mu$ on $X$ is said to be
\emph{shift-invariant} if $\mu((S^g)^{-1}(P))=\mu(P)$ for each $g\in
G$ and each Borel set $P\subseteq X$.  In this case $X,S,\mu$ is a
measure-theoretic dynamical system, and we write
$H(X,\mu,\calP)=H(X,S,\mu,\calP)$, $\ent(X,\mu)=\ent(X,S,\mu)$, etc.
As in \S\ref{sec:sd} it can be shown that
$\ent(X,\mu)=\ent(X,\mu,\calP)$ where $\calP$ is the \emph{canonical
  measurable partition} of $X$, namely $\calP=\{\llb a\rrb\cap X\mid
a\in A\}$.

In the case of an ergodic subshift, there is the following suggestive
characterization of measure-theoretic entropy.

\begin{thm}[Shannon/McMillan/Breiman]
  \label{thm:smb}
  Let $X\subseteq A^G$ be a subshift, and let $\mu$ be an ergodic,
  shift-invariant, probability measure on $X$.  Then for $\mu$-almost
  all $x\in X$ we have
  \[
  \ent(X,\mu)=\lim_{n\to\infty}\frac{\log_2\mu(\llb x\res
    F_n\rrb)}{-|F_n|}\,.
  \]  
\end{thm}

\begin{proof}
  See \cite{smb-amenable}.
\end{proof}

We end this section by noting a significant relationship between
topological entropy and measure-theoretic entropy.

\begin{thm}[Variational Principle]
  \label{thm:var}
  For any subshift $X\subseteq A^G$ we have
  \[
  \ent(X)=\max_\mu\ent(X,\mu)
  \]
  where $\mu$ ranges over ergodic, shift-invariant, probability
  measures on $X$.
\end{thm}

\begin{proof}
  See \nocite{asterisque-40}\cite{mis-Nd} and
  \cite[\S\S16--20]{de-gr-si}.
\end{proof}

\section{Entropy = dimension}
\label{sec:ent-dim}

As in \S\ref{sec:background} let $d$ be a positive integer, let
$G=\NN^d$ or $G=\ZZ^d$, let $A$ be a finite set of symbols, and let
$X\subseteq A^G$ be a subshift.  The purpose of this section is to
prove that $\ent(X)=\dim(X)$.  The special case $G=\NN$ is due to
Furstenberg \cite[Proposition III.1]{furst-tcs-67}.  However, the
general result for $G=\NN^d$ or $G=\ZZ^d$ appears to be new.

As a warm-up for our proof of the general result, we first present
Furstenberg's proof of the special case $G=\NN$.

\begin{thm}[Furstenberg 1967]
  \label{thm:N}
  Let $X\subseteq A^\NN$ be a one-sided subshift.  Then
  $\ent(X)=\dim(X)$.
\end{thm}

\begin{proof}
  By Lemma \ref{lem:easy} we have $\ent(X)\ge\dim(X)$.  To prove
  $\ent(X)\le\dim(X)$ it suffices to prove $\ent(X)\le s$ for all $s$
  such that $\mu_s(X)=0$.  Since $\mu_s(X)=0$ let $\calE$ be such that
  $X\subseteq\bigcup\calE$ and $\sum_{E\in\calE}\diam(E)^s<1$.  As
  noted in \S\ref{sec:dim}, we may safely assume that each $E\in\calE$
  is of the form $E=\llb\sigma\rrb$ where $\sigma\in A^*$, so that
  $\diam(E)=2^{-|\sigma|}$.  By compactness we may assume that $\calE$
  is finite.  Let us write $\calE=\{\llb\sigma\rrb\mid\sigma\in I\}$
  where $I\subset A^*$ is finite.  Let $m=\max\{|\sigma|\mid\sigma\in
  I\}$.  From
  \[
  \sum_{\sigma\in I}2^{-|\sigma|s}=\sum_{E\in\calE}\diam(E)^s<1
  \]
  it follows that
  \[
  \sum_{\sigma_1,\ldots,\sigma_k}2^{-(|\sigma_1|+\cdots+|\sigma_k|)s}
  =\sum_{k=1}^\infty\left(\sum_{\sigma\in I}2^{-|\sigma|s}\right)^k=M<\infty
  \]
  where the first sum is taken over all nonempty finite sequences
  $\sigma_1,\ldots,\sigma_k\in I$.

  The previous paragraph applies to any subshift.  We now bring in the
  special assumption $G=\NN$.  Because $G=\NN$ and
  $F_n=\{0,1,\ldots,n\}$, each $x\in A^G$ is an infinite sequence of
  symbols in $A$, and each $\sigma\in A^*=\bigcup_{n=0}^\infty
  A^{F_n}$ is a nonempty finite sequence of symbols in $A$.  Thus,
  given $x\in X$, we can recursively define an infinite sequence
  $\sigma_1,\ldots,\sigma_k,\ldots\in I$ such that
  $S^{|\sigma_1|+\cdots+|\sigma_{k-1}|}(x)\in\llb\sigma_k\rrb$ for all
  $k$, and then $x=\sigma_1\cat\cdots\cat\sigma_k\cat\cdots$ where
  $\cat$ denotes concatenation of finite sequences.  Now, given
  $n\ge0$, let $k$ be as small as possible such that $x\res
  F_n\subseteq\sigma_1\cat\cdots\cat\sigma_k$.  We then have
  \begin{equation}
    \label{eq:Fnm}
    |F_n|\le|\sigma_1|+\cdots+|\sigma_k|<|F_n|+m
  \end{equation}
  and $\llb x\res
  F_n\rrb\supseteq\llb\sigma_1\cat\cdots\cat\sigma_k\rrb$.  Since the
  sets $\llb\xi\rrb$ for $\xi\in X\res F_n$ are pairwise disjoint, it
  follows that $|X\res F_n|$ is less than or equal to the number of
  finite sequences $\sigma_1,\ldots,\sigma_k\in I$ such that
  (\ref{eq:Fnm}) holds.  For each such finite sequence we have
  $2^{-|F_n|s}<2^{ms}2^{-(|\sigma_1|+\cdots+|\sigma_k|)s}$, so by
  summing over all such finite sequences we obtain $|X\res
  F_n|2^{-|F_n|s}<2^{ms}M$.  Thus $|X\res F_n|2^{-|F_n|s}$ is bounded
  as $n\to\infty$.  It follows by Lemma \ref{lem:entX} that
  $\ent(X)\le s$, Q.E.D.
\end{proof}

We now generalize Furstenberg's result.

\begin{thm}
  \label{thm:G}
  Let $G=\NN^d$ or $G=\ZZ^d$ where $d$ is a positive integer.  Let $A$
  be a finite nonempty set of symbols, and let $X\subseteq A^G$ be a
  subshift.  Then $\ent(X)=\dim(X)$.
\end{thm}

\begin{proof}
  By Lemma \ref{lem:easy} we have $\ent(X)\ge\dim(X)$.  To prove
  $\ent(X)\le\dim(X)$ it suffices to prove $\ent(X)\le s$ for all $s$
  such that $\mu_s(X)=0$.  Using $\mu_s(X)=0$ and the compactness of
  $X$, we can find finite sets $I_l\subset A^*$ for $l=1,2,\ldots$
  such that $X\subseteq\bigcup_{\sigma\in I_l}\llb\sigma\rrb$ and
  $\sum_{\sigma\in I_l}2^{-|\sigma|s}<2^{-l}$ and $|\sigma|<<|\tau|$
  for all $\sigma\in I_l$ and all $\tau\in I_{l+1}$.  Let
  $I_\infty=\bigcup_{l=1}^\infty I_l$.  We have
  \[
  \sum_{\sigma\in
    I_\infty}2^{-|\sigma|s}<\sum_{l=1}^\infty2^{-l}=1
  \]
  hence
  \[
  \sum_{\sigma_1,\ldots,\sigma_k}2^{-(|\sigma_1|+\cdots+|\sigma_k|)s}
  =\sum_{k=1}^\infty\left(\sum_{\sigma\in
      I_\infty}2^{-|\sigma|s}\right)^k=M<\infty
  \]
  where the first sum is taken over all nonempty finite sequences
  $\sigma_1,\ldots,\sigma_k\in I_\infty$.

  For all $\sigma\in A^*$ and all $g\in G$, let $\sigma^g=$ the
  $g$-\emph{translate} of $\sigma$, i.e., $\dom(\sigma^g)=\{g+h\mid
  h\in\dom(\sigma)\}$ and $\sigma^g(g+h)=\sigma(h)$ for all
  $h\in\dom(\sigma)$.  Note that $|\sigma^g|=|\sigma|$ and
  $\llb\sigma^g\rrb=\llb\sigma\rrb^g=(S^g)^{-1}(\llb\sigma\rrb)$.
  Since $X$ is a subshift and $X\subseteq\bigcup_{\sigma\in
    I_l}\llb\sigma\rrb$ for all $l$, we have
  \begin{center}
    $\forall l\,(\forall g\in G)\,(\forall x\in X)\,(\exists\sigma\in
    I_l)\,(x\in\llb\sigma^g\rrb)$.
  \end{center}
  Let $J_\infty=\bigcup_{l=1}^\infty J_l$ where
  $J_l=\{\sigma^g\mid\sigma\in I_l,g\in G\}$.

  \begin{lem}
    \label{lem:vitali}
    Let $\epsilon>0$ be given.  For all sufficiently large $n$ and
    each $x\in X$, we can find a pairwise disjoint set $L\subset
    J_\infty$ such that $\bigcup L\subseteq x\res F_n$ and
    $|\bigcup L|>(1-\epsilon)|F_n|$ and $|L|<\epsilon|F_n|$.
  \end{lem}

  \begin{proof}
    The proof may be viewed as a discrete analog of the classical
    proof of the Vitali Covering Lemma.  Given an ``extremely large''
    configuration $x\res F_n$, we begin by filling in as much of
    $x\res F_n$ as possible with pairwise disjoint ``very very large''
    configurations from $J_\infty$.  After that, we fill in the gaps
    with pairwise disjoint ``very large'' configurations from
    $J_\infty$.  After that, we fill in the remaining gaps with
    pairwise disjoint ``large'' configurations from $J_\infty$.  Et
    cetera.

    Specifically, let $l$ be so large that $(1-(1/4)^d)^l<\epsilon$
    and $1<\epsilon|\sigma|$ for all $\sigma\in I_l$, and let $n$ be
    so large that $n>>|\sigma|$ for all $\sigma\in I_{2l-1}$.  Given
    $x\in X$, let $\xi=x\res F_n$ and let $K_1=\{\tau\in
    J_{2l-1}\mid\tau\subset\xi\}$.  Note that $|\bigcup
    K_1|\ge(3/4)^d|\xi|$, because $|\tau|<<n$ for all $\tau\in
    J_{2l-1}$.  Let $L_1\subseteq K_1$ be pairwise
    disjoint\footnote{Here are the details.  Define
      $L_1=\{\upsilon_j\mid j=1,2,\ldots\}$ where $\upsilon_j\in K_1$
      is chosen inductively so that $\upsilon_i\cap
      \upsilon_j=\emptyset$ for all $i<j$ and $|\upsilon_j|$ is as
      large as possible.  Then for all $\tau\in K_1$ there exists
      $\upsilon\in L_1$ such that $\tau\cap\upsilon\ne\emptyset$ and
      $|\tau|\le|\upsilon|$.  From this it follows that $|\bigcup
      L_1|\ge|\bigcup K_1|/3^d$.}  such that $|\bigcup L_1|\ge|\bigcup
    K_1|/3^d$.  It follows that $|\bigcup L_1|\ge|\xi|/4^d$, hence
    $|\xi\setminus\bigcup L_1|\le(1-(1/4)^d)|\xi|$.  If
    $|\xi\setminus\bigcup L_1|\le(1-(1/4)^d)^2|\xi|$, let
    $L_2=K_2=\emptyset$.  Otherwise, let $K_2=\{\tau\in
    J_{2l-2}\mid\tau\subset\xi\setminus\bigcup L_1\}$ and note that
    $|\bigcup K_2|\ge(3/4)^d|\xi\setminus\bigcup L_1|$, because
    $|\tau|<<|\upsilon|$ for all $\tau\in J_{2l-2}$ and all
    $\upsilon\in L_1$.  As before let $L_2\subseteq K_2$ be pairwise
    disjoint such that $|\bigcup L_2|\ge|\bigcup K_2|/3^d$.  It
    follows as before that $|\xi\setminus\bigcup(L_1\cup
    L_2)|\le(1-(1/4)^d)^2|\xi|$.  Continuing in this fashion for $l$
    steps, we obtain $L_1\subseteq J_{2l-1}$ and $L_2\subseteq
    J_{2l-2}$ and \ldots\ and $L_l\subseteq J_l$ such that
    $|\xi\setminus\bigcup(L_1\cup\cdots\cup
    L_l)|\le(1-(1/4)^d)^l|\xi|$.  Finally let $L=L_1\cup\cdots\cup
    L_l$.  By construction $L$ is pairwise disjoint and $\bigcup
    L\subseteq\xi$.  Moreover $|\xi\setminus\bigcup
    L|\le(1-(1/4)^d)^l|\xi|<\epsilon|\xi|=\epsilon|F_n|$, hence
    $|\bigcup L|>(1-\epsilon)|\xi|=(1-\epsilon)|F_n|$.  For each
    $\tau\in L$ we have $1<\epsilon|\tau|$, hence
    $|L|<\epsilon|\bigcup L|\le\epsilon|F_n|$.  This proves Lemma
    \ref{lem:vitali}.
  \end{proof}  

  \begin{lem}
    \label{lem:count}
    Let $\epsilon$ and $n$ be as in Lemma \ref{lem:vitali}.  Then
    $|X\res F_n|$ is less than or equal to $(|A|+1)^{2\epsilon|F_n|}$
    times the number of sequences $\sigma_1,\ldots,\sigma_k\in
    I_\infty$ such that $|\sigma_k|+\cdots+|\sigma_k|\le|F_n|$.
  \end{lem}

  \begin{proof}
    The idea of the proof is that, by Lemma \ref{lem:vitali}, each
    $x\res F_n\in X\res F_n$ is almost entirely covered by a finite
    sequence of pairwise disjoint translates of elements of
    $I_\infty$.  These elements of $I_\infty$ can be used to give a
    concise description of $x\res F_n$.

    Given $x\in X$ let $L=\{\tau_1,\ldots,\tau_k\}$ be as in the
    conclusion of Lemma \ref{lem:vitali}.  For each $i=1,\ldots,k$ let
    $\sigma_i\in I_\infty$ be such that $\tau_i=\sigma_i^g$ for some
    $g\in G$.  Since $\tau_1,\ldots,\tau_k$ are pairwise disjoint and
    $\bigcup_{i=1}^k\tau_i=\bigcup L\subseteq x\res F_n$, we have
    $|\sigma_1|+\cdots+|\sigma_k|=|\tau_1|+\cdots+|\tau_k|\le|F_n|$.
    Let $\llex$ be the lexicographical ordering of $F_n$.  For each
    $i=1,\ldots,k$ let $g_i=$ the least element of
    $\dom(\tau_i)\subseteq F_n$ with respect to $\llex$.  Reordering
    $\tau_1,\ldots,\tau_k$ as necessary, we may assume that
    $g_1\llex\cdots\llex g_k$.  Let
    $U=F_n\setminus\bigcup_{i=1}^k\dom(\tau_i)$, and let
    $V=U\cup\{g_1,\ldots,g_k\}$.  By Lemma \ref{lem:vitali} we have
    $|U|=|F_n|-|\bigcup L|<\epsilon|F_n|$ and $k=|L|<\epsilon|F_n|$,
    hence $|V|=|U|+k\le m$ where $m=2\lfloor\epsilon|F_n|\rfloor$.
    For each $j=1,\ldots,m$ define $a_j\in A\cup\{0\}$ as follows.  If
    $j\le|V|$ let $g$ be the $j$th element of $V$ with respect to
    $\llex$.  If $g\in U$, let $a_j=x(g)$.  Otherwise, let $a_j=0$.
    Clearly $x\res F_n$ can be recovered from the pair of sequences
    $a_1,\ldots,a_m$ and $\sigma_1,\ldots,\sigma_k$.  This proves
    Lemma \ref{lem:count}.
  \end{proof}

  To prove Theorem \ref{thm:G}, let $\epsilon$ and $n$ be as in Lemmas
  \ref{lem:vitali} and \ref{lem:count}.  Because
  $|\sigma_1|+\cdots+|\sigma_k|\le|F_n|$ implies $2^{-|F_n|s}\le
  2^{-(|\sigma_1|+\cdots+|\sigma_k|)s}$, it follows from Lemma
  \ref{lem:count} and the definition of $M$ that
  \[
  |X\res F_n|2^{-|F_n|s}<(|A|+1)^{2\epsilon|F_n|}M\,,
  \]
  i.e.,
  \[
  |X\res F_n|2^{-|F_n|(s+2\epsilon\log_2(|A|+1))}<M\,.
  \]
  Thus $|X\res F_n|2^{-|F_n|(s+2\epsilon\log_2(|A|+1))}$ is bounded as
  $n$ goes to infinity, so by Lemma \ref{lem:entX} we have $\ent(X)\le
  s+2\epsilon\log_2(|A|+1)$.  And this holds for all $\epsilon>0$, so
  $\ent(X)\le s$.  The proof of Theorem \ref{thm:G} is now complete.
\end{proof}

\section{Dimension = complexity}
\label{sec:dim-k}

As before let $d$ be a positive integer, let $G=\NN^d$ or $G=\ZZ^d$,
let $A$ be a finite set of symbols, and let $X\subseteq A^G$ be a
subshift.  In this section we prove that the Hausdorff dimension of
$X$ is equal to the effective Hausdorff dimension of $X$.  In addition
we obtain a sharp characterization of $\dim(X)$ in terms of the
Kolmogorov complexity of finite pieces of the individual orbits of
$X$, i.e., in terms of $\K(x\res F_n)$ for $x\in X$ and
$n=1,2,\ldots$.  Our results apply even when $X$ is not effectively
closed.

\begin{lem}
  \label{lem:limsup}
  For all $x\in X$ we have
  \begin{equation}
    \label{eq:limsup}
    \limsup_{n\to\infty}\frac{\K(x\res
      F_n)}{|F_n|}\,\,\le\,\,\ent(X)\,.
  \end{equation}
\end{lem}

\begin{proof}
  Fix a positive integer $m$.  Given $n\ge m$, let $k$ be a positive
  integer such that $mk\le n<m(k+1)$.  Partitioning $F_{m(k+1)}$ into
  $(k+1)^d$ blocks of size $|F_m|$, we see that $|X\res
  F_n|\le(k+1)^d|X\res F_m|$ and there is a constant $c$ independent
  of $n$ such that $\K(x\res F_n)\le(k+1)^d\log_2|X\res
  F_m|+2\log_2n+c$ for all $x\in X$.  Thus
  \[
  \frac{\K(x\res F_n)}{|F_n|}\le\frac{(k+1)^d\log_2|X\res
    F_m|+2\log_2n+c}{k^d|F_m|}\to\frac{\log_2|X\res F_m|}{|F_m|}
  \]
  as $n\to\infty$.  Since this holds for all $m$, we now see that
  (\ref{eq:limsup}) follows from (\ref{eq:entX}).
\end{proof}

\begin{lem}
  \label{lem:lim}
  For some $x\in X$ we have
  \begin{equation}
    \label{eq:lim}
    \lim_{n\to\infty}\frac{\K(x\res F_n)}{|F_n|}\,\,=\,\,\ent(X)\,.
  \end{equation}
\end{lem}

\begin{proof} 
  By the Variational Principle \ref{thm:var} let $\mu$ be an ergodic,
  shift-invariant, probability measure on $X$ such that
  $\ent(X,\mu)=\ent(X)$.  Fix $s<\ent(X)$.  Let
  \[
  D_n=\left\{\xi\in A^{F_n}\bigm|\K(\xi)<|F_n|s\right\}\,.
  \]
  Clearly $|D_n|\le2^{|F_n|s}$.  Fix $\epsilon>0$ such that
  $s+\epsilon<\ent(X)$, and let
  \[
  T_n=\{\xi\in A^{F_n}\mid\mu(\llb\xi\rrb)<2^{-|F_n|(s+\epsilon)}\}\,.
  \]
  The Shannon/McMillan/Breiman Theorem \ref{thm:smb} tell us that for
  $\mu$-almost all $x\in X$ and all sufficiently large $n$ we have
  \[
  \frac{\log_2\mu(\llb x\res F_n\rrb)}{-|F_n|}>s+\epsilon\,,
  \]
  i.e., $x\res F_n\in T_n$, i.e., $x\in\llb T_n\rrb$.  On the other
  hand, for each $n$ we have
  \begin{center}
    $\mu(\llb D_n\rrb\cap\llb T_n\rrb)=\mu(\llb D_n\cap
    T_n\rrb)\le2^{|F_n|s}2^{-|F_n|(s+\epsilon)}=2^{-|F_n|\epsilon}$
  \end{center}
  and so
  \[
  \sum_{n=1}^\infty\mu(\llb D_n\rrb\cap\llb T_n\rrb)<\infty\,.
  \]
  Thus the Borel/Cantelli Lemma tells us that, for $\mu$-almost all
  $x$ and all sufficiently large $n$, $x\notin\llb D_n\rrb\cap\llb
  T_n\rrb$.  But then it follows that, for $\mu$-almost all $x$ and
  all sufficiently large $n$, $x\notin\llb D_n\rrb$, i.e., $x\res
  F_n\notin D_n$, i.e., $\K(x\res F_n)\ge|F_n|s$.  Since this holds
  for all $s<\ent(X)$, we now see that (\ref{eq:lim}) holds for
  $\mu$-almost all $x\in X$.  This completes the proof.
\end{proof}

\begin{thm}
  \label{thm:dim-k}
  Let $G=\NN^d$ or $G=\ZZ^d$ where $d$ is a positive integer.  Let $A$
  be a finite set of symbols, and let $X\subseteq A^G$ be a subshift.
  Then
  \[
  \ent(X)=\dim(X)=\effdim(X)\,.
  \]
  Moreover
  \[
  \dim(X)\ge\limsup_{n\to\infty}\frac{\K(x\res F_n)}{|F_n|}
  \]
  for all $x\in X$, and
  \[
  \dim(X)=\lim_{n\to\infty}\frac{\K(x\res
    F_n)}{|F_n|}
  \]
  for some $x\in X$.
\end{thm}

\begin{proof}
  This follows from Theorems \ref{thm:effdim-k} and \ref{thm:G} and
  Lemmas \ref{lem:limsup} and \ref{lem:lim}.
\end{proof}

\begin{ques}
  {\ }
  \begin{enumerate}
  \item Can we find an ``elementary'' or ``direct'' proof of Lemma
    \ref{lem:lim}?  I.e., a proof which does not use measure-theoretic
    entropy?
  \item Is it possible to generalize Theorems \ref{thm:G} and
    \ref{thm:dim-k} so as to apply to wider classes of groups or
    semigroups?  For example, do Theorems \ref{thm:G} and
    \ref{thm:dim-k} continue to hold if $G$ is an amenable group
    \cite{amenable-survey}?
  \item Is it possible to generalize Theorems \ref{thm:G} and
    \ref{thm:dim-k} so as to apply to scaled entropy and scaled
    Hausdorff dimension?  For example, what about
    \[
    \liminf_{n\to\infty}\,\frac{\K(x\res F_n)}{\sqrt{|F_n|}}\,?
    \]
  \end{enumerate}
\end{ques}

\addcontentsline{toc}{section}{References}

\end{document}